\documentclass[reqno]{amsart}

\usepackage{amsmath, amssymb, amsthm, amsfonts, color, mathrsfs}

\usepackage[dvipsnames]{xcolor}

\usepackage{ytableau, caption}

\usepackage[margin = 1.4in]{geometry}
\usepackage{adjustbox}

\input xy
\xyoption{all}

\usepackage{hyperref}

\usepackage{tikz-cd}

\numberwithin{equation}{subsection}

\newtheorem{thm}[subsubsection]{Theorem}
\newtheorem{lemma}[subsubsection]{Lemma}

\theoremstyle{definition}
\newtheorem{defn}[subsubsection]{Definition}

\newtheorem{remark}[subsubsection]{Remark}
\newtheorem{example}[subsubsection]{Example}

\usepackage[utf8]{inputenc}
\usepackage[T1]{fontenc}
\usepackage{xcolor} 
\usepackage[most]{tcolorbox} 

\definecolor{humanbubble}{RGB}{235, 245, 255} 
\definecolor{aibubble}{RGB}{245, 245, 245}    
\definecolor{loggray}{RGB}{230, 230, 230}
\definecolor{indexbg}{RGB}{248, 248, 248}

\newtcolorbox{interactionlog}[2][]{
  enhanced,
  arc=0pt, outer arc=0pt,
  colback=white, colframe=black!60,
  boxrule=0.8pt,
  fonttitle=\bfseries\sffamily, coltitle=black, colbacktitle=loggray,
  title={Human-AI Interaction Card \if\relax\detokenize{#1}\relax\else for #1\fi},
  halign title=center, attach title to upper,
  after title={\vspace{4pt}\hrule\vspace{10pt}},
  lower separated=true,
  segmentation style={solid, black!60, line width=0.8pt},
  colbacklower=indexbg,
  after upper={\par\vfill
    \begin{tcolorbox}[
      enhanced, colback=indexbg, colframe=white, boxrule=0pt,
      top=0pt, bottom=0pt, fontupper=\footnotesize\sffamily,
      title=Raw prompts and outputs, coltitle=black!70, attach title to upper,
      after title={:\enskip}, sharp corners
    ]
    #2 
    \end{tcolorbox}
  }
}

\newcommand{\human}[1]{%
  \noindent\begin{flushright}
    \begin{minipage}[c]{0.70\textwidth}
      \begin{tcolorbox}[
        enhanced,
        colback=humanbubble, colframe=black!15,
        arc=6pt, sharp corners=southeast, boxrule=0.5pt,
        left=6pt, right=6pt, top=4pt, bottom=4pt, boxsep=0pt
      ]\small #1\end{tcolorbox}
    \end{minipage}%
    \hspace{8pt}
    \begin{minipage}[c]{40pt}
      \footnotesize\sffamily\textbf{Human}
    \end{minipage}
  \end{flushright}
}

\newcommand{\ai}[2]{%
  \noindent\begin{flushleft}
    \begin{minipage}[c]{45pt}
      \footnotesize\sffamily\textbf{#1}
    \end{minipage}%
    \hspace{2pt}
    \begin{minipage}[c]{0.70\textwidth}
      \begin{tcolorbox}[
        enhanced,
        colback=aibubble, colframe=black!15,
        arc=6pt, sharp corners=southwest, boxrule=0.5pt,
        left=6pt, right=6pt, top=4pt, bottom=4pt, boxsep=0pt
      ]\small #2\end{tcolorbox}
    \end{minipage}
  \end{flushleft}
}

\def\BB{\mathbb{B}}

\def\VV{\mathbb{V}}

\def\XX{\mathbb{X}}

\newcommand\cD{\mathcal{D}}
\newcommand\cE{\mathcal{E}}

\newcommand\cZ{\mathcal{Z}}

\newcommand\frR{\mathfrak{R}}

\newcommand{\sL}{{\mathscr{L}}}

\newcommand{\Bun}{\textup{Bun}}

\newcommand\ev{\textup{ev}}

\newcommand{\Gr}{\textup{Gr}}

\newcommand\Sht{\textup{Sht}}

\newcommand\Sym{\textup{Sym}}

\newcommand{\vol}{\textup{vol}}

\DeclareMathOperator\Pf{Pf}

\newcommand\GL{\textup{GL}}

\newcommand\SO{\textup{SO}}

\DeclareMathOperator\Spin{Spin}

\renewcommand{\j}[1]{\langle{#1}\rangle}

\newcommand\quash[1]{}

\newcommand{\ov}{\overline}

\newcommand{\mf}[1]{\mathfrak{#1}}

\newcommand{\pderiv}[2]{\frac{\partial #1}{\partial #2}}

\renewcommand\c\circ

\newcommand{\ari}{\ar@{^{(}->}} 

\renewcommand\a\alpha
\renewcommand\b\beta
\newcommand\g\gamma
\newcommand\G\Gamma
\renewcommand\d\delta
\newcommand\D\Delta

\renewcommand\r\rho

\newcommand{\y}{\eta}

\usepackage{color}

\newcommand{\Aletheia}{\emph{Aletheia}}

\newcommand{\F}{\mathbf{F}}

\newcommand{\Q}{\mathbf{Q}}

\newcommand{\ol}[1]{\overline{#1}}

\newcommand{\co}{\colon}

\DeclareMathOperator{\Id}{Id}

\RequirePackage{xspace}

\newcommand{\rH}{\ensuremath{\mathrm{H}}\xspace}

\DeclareMathOperator{\PSO}{PSO}
\DeclareMathOperator{\PSp}{PSp}
\DeclareMathOperator{\SSYT}{SSYT}
\DeclareMathOperator{\len}{len}
\DeclareMathOperator{\height}{height}

\title[]{Eigenweights for arithmetic Hirzebruch Proportionality}

\dedicatory{}
\author[Tony Feng]{Tony Feng}
        \thanks{The mathematical content of this paper was fully generated by a math-research agent powered by Gemini Deep Think, internally codenamed {\Aletheia} at Google DeepMind (see \S \ref{sec:origin-story}). Aside from building \Aletheia, the author's contribution was only to rewrite the mathematical content into paper form.}
\address{Department of Mathematics, University of California Berkeley, Berkeley, CA 94720}
\email{fengt@berkeley.edu}

\begin{document}

\begin{abstract}
Prior work of Feng--Yun--Zhang \cite{FYZ4} established a (Higher) Arithmetic Hirzebruch Proportionality Principle, expressing the arithmetic volumes of  moduli stacks of shtukas in terms of differential operators applied to $L$-functions. This formula involves certain ``eigenweights'' which were calculated in simple cases by Feng--Yun--Zhang, but not in general. We document work of a (custom) AI Agent built upon Gemini Deep Think, which employs tools from algebraic combinatorics to connect these eigenweights to the representation theory of symmetric groups, and then determines them for all classical groups.
\end{abstract}

\maketitle

\tableofcontents

\section{Introduction} 
\subsection{Arithmetic Hirzebruch Proportionality} The famous \emph{Hirzebruch Proportionality Principle} \cite{Hir58} expresses the Chern numbers of an automorphic vector bundle on a compact locally symmetric space as a multiple of the corresponding Chern numbers on the compact dual variety, with proportionality constant given by the value of an $L$-function. Mumford \cite{Mum77} generalized Hirzebruch's formula to non-compact locally symmetric spaces. 

Recently, joint work of the author with Yun and Zhang in \cite{FYZ4} investigated an extension of this principle, which we call ``Arithmetic Hirzebruch Proportionality''. It relates the ``arithmetic volume'' of Chern classes on arithmetic moduli spaces to \emph{differential operators} applied to an $L$-function. This principle manifests in both the number field setting (where it is closely related to classical Hirzebruch(--Mumford) Proportionality) and the function field setting. The precise differential operator appearing in Arithmetic Hirzebruch Proportionality is governed by certain fundamental structure constants that we call \emph{eigenweights}. These were calculated for some examples in \cite{FYZ4}, and the purpose of this paper is to determine them in general.  

\subsubsection{An example of Arithmetic Hirzebruch Proportionality}  For concreteness, we illustrate a special case of the main result of \cite{FYZ4}. Let $X$ be a smooth, projective, geometrically connected curve of genus $g$ over $\F_q$, and $G$ be a semisimple reductive group of rank $n$ over $X$. To $G$ we can associate a multivariable $L$-function $\sL_{X, G}(s_1, \ldots, s_n)$; when $s_1 = \ldots = s_n = s$, this recovers the ``$L$-function of the motive of $G$'' in the sense of Gross \cite{Gro97}.

Let $\mu$ be a minuscule dominant coweight of $G$, $r \geq 0$, and $\Sht_G^{\mu^r}$ be the moduli stack of $G$-shtukas on $X$ with $r$ legs of type $\mu$. For $P_\mu$ the corresponding parabolic subgroup of $G$, pick a cohomology class $\eta \in \rH^{2N}(\BB P_\mu)$ where $N = \dim G/P_\mu + 1$. The moduli stack $\Sht_G^{\mu^r}$ carries $r$ tautological ``Hodge'' bundles, and upon evaluating $\eta$ on them we obtain tautological characteristic classes $\ev_i^*(\eta)$ for $i=1, \ldots, r$. Then $\prod_{i=1}^{r}(\ev_{i}^{*}\eta)$ is a top cohomology class on $\Sht^{\mu^r}_G$, and \cite[Theorem 5.5.9]{FYZ4} implies that
\begin{equation}\label{eq:arithmetic-volume}
\vol(\Sht^{\mu^r}_{G}, \prod_{i=1}^{r}\ev_{i}^{*}\y)= \# \pi_0(\Bun_G) \cdot q^{(g-1) \dim G} \cdot \mf d^r \Big|_{s_1=\ldots = s_n = 0} \sL_{X,G}(s_1, \ldots, s_n)
\end{equation}
where $\mf d$ is the differential operator
\begin{equation}\label{eq:diff-op}
\mf d = (- \log q)^{-1} \sum_{i=1}^n \epsilon_i(\eta, \mu) \pderiv{}{s_i}.
\end{equation}
 The weighting constants $\epsilon_i(\eta, \mu)$ in \eqref{eq:diff-op} are what we call \emph{eigenweights}, as they are computed as the eigenvalues of a certain ``local'' operator $\ol \nabla^\eta_\mu$. 

\begin{remark}
The original motivation for this direction of research was to extend the \emph{Higher Siegel--Weil formula} of \cite{FYZ} to the ``singular terms''. For example, when $G$ is a unitary group, part of the special cycle  $\cZ_{\cE}^r(0)$ constructed in \cite[\S 4]{FYZ2} can be interpreted as an arithmetic volume, and the matching derivative of the $L$-function is part of the constant term of the Siegel--Eisenstein series. In fact, such arithmetic volumes appear in more general singular coefficients; the corank-one Higher Siegel--Weil formula recently established in \cite{FHM25} crucially uses a special case of the result \eqref{eq:arithmetic-volume}.
\end{remark}

\subsubsection{Arithmetic volume of Shimura varieties} 
In the number field setting, \emph{arithmetic volume} is interpreted as an Arakelov degree of automorphic vector bundles on Shimura varieties. This has been computed in special cases, such as in \cite{BH24, Koh05, Hor14}. Our perspective gives a uniform explanation for the constants arising in the resulting formulas (e.g., \cite[Conjecture 5.8]{Koh05}). We expect to revisit this in a future work (with Yun and Zhang). 

\subsection{Definition of eigenweights}
Let $G$ be a split reductive group over $\Q$, $T \subset G$ be a split maximal torus, and $\XX^*(T)$ its character group. All cohomology groups will be Betti cohomology formed with rational coefficients. Consider the graded ring 
\[
R := \rH^*(\BB T) = \Sym_{\Q}(\XX^*(T)_{\Q})
\]
where $\XX^*(T)$ is placed in degree $2$. This carries an obvious action of the Weyl group $W$, and the invariant ring $R^W$ identifies with $\rH^*(\BB G)$. Let $I \subset R^W$ be the augmentation ideal, and $\VV := I/I^2$ be the first associated graded for the augmentation filtration on $R^W$. This is a slight renormalization of Gross's ``motive of $G$'' \cite{Gro97}; see \cite[Remark 4.1.9]{FYZ4} for the comparison. The \emph{eigenweights} are defined in \cite[\S 5.3]{FYZ4} as the collection of eigenvalues for a certain operator $\ov \nabla^\eta_\mu$ on $\VV$; to keep this paper self-contained, we proceed to recall the definition. 

Let $\mu \in \XX_*(T)$ be a minuscule coweight of $G$, and $L_\mu \subset P_\mu$ be the associated Levi and parabolic subgroups of $G$. Let $W_\mu$ be the Weyl group of $L_\mu$. The map $\BB P_\mu \rightarrow \BB G$ is a fiber bundle for $G/P_\mu$, so it induces a pushforward map on cohomology, which we denote
$$
\int_{G/P_\mu}  : R^{W_\mu} = \rH^*(\BB P_\mu) \rightarrow \rH^*(\BB G) = R^W.
$$

Since the coweight $\mu$ can be viewed in $\XX_*(T)$ while $R = \mathrm{Sym}_{\Q}(\XX^*(T)_{\Q})$, the partial derivative $\partial_\mu \colon R \rightarrow R$ is defined, and it carries $R^W \rightarrow R^{W_\mu}$.

\begin{defn}
Let $\eta \in R^{W_\mu}$. The associated operator $\ol {\nabla}^\eta_\mu$ on $\VV$ is defined by the formula
$$
\ol {\nabla}^\eta_\mu(f) = \int_{G/P_\mu} \eta \cdot  \partial_{\mu} (f).
$$
The \emph{eigenweights} are the eigenvalues of the operator $\ol {\nabla}^\eta_\mu$ on $\VV$. 
\end{defn}

\begin{remark}\label{rem:Omega-eigenweight}
There is a particular format of $\eta$ which is most natural for the application to arithmetic volumes in \cite{FYZ4}, and is therefore the only type of $\eta$ that we consider. Namely, we begin with a Casimir element $\Omega \in \Sym^2(\XX^*(T)_{\Q})^W$; for example, if $G$ is semisimple then $\Omega$ is a $\Q$-multiple of the Killing form (restricted to $\XX^*(T)_{\Q}$). Then we set $t_\mu := \partial_\mu (\Omega)$, and $\eta := t_\mu^{1 + \dim G/P_\mu}$. (The exponent is arranged so that $\ol {\nabla}^\eta_\mu$ preserves the grading, otherwise the eigenweights would automatically vanish). In this situation, we write $\epsilon_k(\Omega, \mu)$ for the eigenweights, indexed by a basis of eigenvectors for $\VV$.\footnote{Note that this notation differs from \cite{FYZ4}, where the eigenweights are instead denoted $\epsilon_k(\eta, \mu)$.}
\end{remark}

\subsection{Calculation of eigenweights} In \cite{FYZ4}, the eigenweights are calculated in several examples. For partitions or coweights, we use the notation
\[
(\ldots, \pi_i^{e_i}, \ldots) = (\ldots, \underbrace{\pi_i, \ldots, \pi_i}_{e_i \text{ times}}, \ldots).
\]
For groups of rank $n$, we have $R = \Q[x_1, \ldots, x_n]$ and we take the Casimir element 
\[
\Omega := \frac{1}{2} \sum_{i=1}^n x_i^2 \in R^W.
\]
As we are interested in closed form descriptions of eigenweights, we do not discuss exceptional groups; among them, only $E_6$ and $E_7$ admit minuscule coweights, and in those cases the eigenweights can be calculated by a finite algorithm.

\subsubsection{Type A} Let $G = \GL_n$. The minuscule coweights are $(1^m, 0^{n-m})$ for $1 \leq m < n$. A basis for $\VV$ consists of the power sums $p_1, \ldots, p_n$ in the $\{x_1, \ldots, x_n\}$, and we denote $\epsilon_k(\Omega, \mu)$ the eigenweight associated to $p_k$. 

 For $n \geq 2$ and $\mu = (1, 0^{n-1})$,\cite[Proposition 4.4.1]{FYZ4} shows that 
\begin{equation}\label{eq:FYZ-epsilon-1}
\epsilon_k(\Omega, \mu) = (-1)^{n-1} \text{ for $k=1, 2, \ldots, n$}.
\end{equation}
For $n \geq 3$ and $\mu = (1^2, 0^{n-2})$, \cite[Proposition 5.4.1]{FYZ4} shows that 
\begin{equation}\label{eq:FYZ-epsilon-2}
\epsilon_k(\Omega, \mu) = \frac{1}{n} \binom{2n-2}{n-1}  - \binom{2n-3}{n-k} + 2 \binom{2n-3}{n-k-1} - \binom{2n-3}{n-k-2} \text{ for $k=1, 2, \ldots, n$}.
\end{equation}
For $\mu = (1^m, 0^{n-m})$ with $m \geq 3$, \cite{FYZ4} did not calculate the eigenweights in closed form. This is addressed by Theorem \ref{thm:A-main} below. We shall see that the general results are much more complex than the examples treated already in \cite{FYZ4}. However, there \emph{is} a clean, uniform answer in terms of the representation theory of symmetric groups. To formulate it, recall that for the symmetric group $S_r$, 
both \begin{itemize}
\item irreducible representations of $S_r$, and \item conjugacy classes of $S_r$
\end{itemize}
are naturally indexed by partitions of $r$. For a partition $\pi$ of $r$, let $\chi^\pi$ be the character of the corresponding irreducible representation. For another partition $\nu$ of $r$, we write $\chi^\pi(\nu)$ for the value of $\chi^\pi$ on the conjugacy class of $S_r$ indexed by $\nu$. Below, addition of partitions is entrywise, with the shorter partition padded by zeros. 

\begin{thm}\label{thm:A-main}
Let $G = \GL_n$ and fix $1 \leq m < n$. Consider the minuscule coweight $\mu = (1^m, 0^{n-m})$. Let $N = m(n-m)+1$ be the arithmetic dimension of $G/P_\mu$. For $\Omega := \frac{1}{2} \sum_{i=1}^n x_i^2$, the eigenweights are\footnote{As usual, the use of exponents in partitions indicates repetitions, e.g., $(n-m)^m$ means $(\underbrace{n-m, n-m, \ldots, n-m}_{m \text{ times}})$.}
    \begin{equation}\label{eq:A-eigenvalues} 
\epsilon_k(\Omega, \mu) = (-1)^{N-1} \Lambda_k \sum_{j=0}^{\min(k-1, m-1)} (-1)^j \chi^{\pi_j(k) + (n-m)^m}(\nu_k ) \quad \text{ for $k = 1, \ldots, n$}
    \end{equation}
    where $\pi_j(k)$ is the partition $(k-j, 1^j)$ and $\nu_k$ is the partition $(k-1, 1^{N}) $, and $\Lambda_k := \begin{cases} m  & k=1, \\ 1 & k > 1. \end{cases}$ 
    \end{thm}

\begin{remark}
For any specific choice of parameters, the quantity \eqref{eq:A-eigenvalues} is easy to compute algorithmically using the Murnaghan--Nakayama Rule \cite[\S 7.17]{Stan2} (recalled in \S \ref{ssec:MN-rule}). In particular, we spell out in \S \ref{ssec:a-special-cases} how to recover the results \eqref{eq:FYZ-epsilon-1} and \eqref{eq:FYZ-epsilon-2} of \cite{FYZ4} for the special cases $m=1$ and $m=2$.
\end{remark}

\subsubsection{Type B} For $G = \SO_{2n+1}$, there is a unique minuscule coweight, represented by $\mu = (1, 0^{n-1})$. A basis for $\VV$ consists of the power sum polynomials $p_1, \ldots, p_n$ in the $\{x_1^2, \ldots, x_n^2\}$, and we denote $\epsilon_k(\Omega, \mu)$ the eigenweight associated to $p_k$. In \cite[\S 5.2]{FYZ4}, the eigenweights for $\Omega = \frac{1}{2}\sum_{i=1}^n x_i^2$ are calculated to be
\[
\epsilon_k(\Omega, \mu) = -4 \text{ for } k= 1, 2, \ldots, n. 
\]

\subsubsection{Type C} For $G = \PSp_{2n}$, there is a unique minuscule spin coweight, corresponding to the spin representation of the dual group $\Spin_{2n+1}$. For this $\mu$, \cite{FYZ4} did not calculate the associated eigenweights. This is addressed by our next theorem. A basis for $\VV$ consists of the power sum polynomials $p_1^{(2)}, \ldots, p_n^{(2)}$ in the $\{x_1^2, \ldots, x_n^2\}$, and we denote $\epsilon_k(\Omega, \mu)$ the eigenweight associated to $p_k^{(2)}$.

\begin{thm}\label{thm:C-main}
Fix $n \geq 2$ and let $G=\PSp_{2n}$, and $\rho_n$ be the partition $(n, n-1, \ldots, 1)$. Let $\mu$ be the minuscule (spin) coweight of $G$ and $N = \binom{n+1}{2}+1$ be the arithmetic dimension of $G/P_\mu$. For $\Omega := \frac{1}{2}\sum_{i=1}^n x_i^2$, the eigenweights are
\begin{equation}\label{eq:C-eps-k}
\epsilon_k(\Omega, \mu) =  (-1)^{N-1} 2^{-N} \sum_{j=0}^{k-1} (-1)^j \chi^{2\pi_j(k)+\rho_n}(\nu_k) \quad \text{for $k=1, 2, \ldots, n$},
\end{equation}
where $\pi_j(k) = (k-j, 1^j)$ as in Theorem \ref{thm:A-main}, and $\nu_k = (2k-1, 1^N)$.
\end{thm}

\subsubsection{Type D} For $G = \PSO_{2n}$, there are three minuscule coweights. A basis for $\VV$ consists of the power sum polynomials $p_1^{(2)}, \ldots, p_{n-1}^{(2)}$ in the $\{x_1^2, \ldots, x_n^2\}$ as well as the Pfaffian $\Pf := x_1  x_2 \cdots  x_n$, and we denote $\epsilon_k(\Omega, \mu)$ (resp. $\epsilon_{\Pf}$) the eigenweight associated to $p_k^{(2)}$ (resp. $\Pf$). The \emph{standard} coweight (corresponding to the standard representation of $\SO_{2n}$) is represented by $\mu = (1, 0^{n-1})$. For $\Omega = \frac{1}{2}\sum_{i=1}^n x_i^2$, \cite{FYZ4} calculates its eigenweights to be 
\begin{equation}\label{eq:D_n-standard-eigenweights}
\epsilon_k(\Omega, \mu) = 4 \text{ for } k= 1, 2, \ldots, n-1 \text{ and } \epsilon_{\Pf} = 2.
\end{equation}
There are also two \emph{spin} coweights, corresponding to the two spin representations of $\Spin_{2n}$, which \cite{FYZ4} did not calculate in closed form. This is addressed by our next Theorem. 

\begin{thm}\label{thm:D-main}
Fix $n \geq 2$ and let $G = \PSO_{2n}$ and $\delta_n$  be the partition $(n-1, \ldots, 1,0)$. Let $\mu$ be either of the spinor minuscule coweights of $G$, and $N := \binom{n}{2}+1$ be the arithmetic dimension of $G/P_\mu$. Take $\Omega = \frac{1}{2}\sum x_i^2$.

(1) If $n$ is odd, then the eigenweights are 
\begin{equation}\label{eq:D-eps-k}
\epsilon_k(\Omega, \mu) = (-1)^{N-1} 2^{n-1-N} \sum_{j=0}^{k-1} (-1)^j \chi^{2 \pi_j(k) + \delta_n}(\nu_k) \quad \text{for $k = 1, 2, \ldots,n-1$}
\end{equation}
and\footnote{The factor $\frac{ \binom{n}{2}!}{\prod_{i=1}^{n-1}(2i-1)^{n-i}}$ in \eqref{eq:eps-pfaff} arises as the dimension of the irreducible representation of the symmetric group associated to the partition $\delta_n$.}
\begin{equation}\label{eq:eps-pfaff}
\epsilon_{\Pf}(\Omega, \mu) = (-1)^{N-1} 2^{n-2-N} \left( \binom{n}{2}+1 \right) \frac{ \binom{n}{2}!}{\prod_{i=1}^{n-1}(2i-1)^{n-i}}.
\end{equation}
where $\pi_j(k) = (k-j, 1^j)$ and $\nu_k = (2k-1, 1^{N})$ as in Theorem \ref{thm:C-main}. 

(2) If $n = 2m$ is even, then for $k \neq m$ the eigenweight attached to $p_k^{(2)}$ is the same expression \eqref{eq:D-eps-k}. The elements $p_m^{(2)}$ and $\Pf$ are not eigenvectors for $\ol {\nabla}^\eta_\mu$, but we calculate the action of $\ol {\nabla}^\eta_\mu$ on their 2-dimensional span, in \eqref{eq:D-even-matrix}. 
\end{thm}

\begin{remark}
Theorem \ref{thm:D-main}(2) addresses a question raised in \cite{FYZ4}, namely whether the $\ol {\nabla}^\eta_\mu$ commute as $\mu$ ranges over different minuscule coweights, when $G$ is semisimple. This property is of interest because it is an assumption needed for the main result of \cite{FYZ4} (in all types other than $D_{2m}$, it is implied by formal considerations). Our calculations show that the commutativity does \emph{not} necessarily hold in type $D_{2m}$. For the standard minuscule coweight $\mu = \omega_1^\vee$ of $\PSO_{4m}$, \cite{FYZ4} showed that $p_m^{(2)}$ and $\Pf$ are eigenvectors for $\ol {\nabla}^\eta_{\omega_1^\vee}$, with distinct eigenvalues. However, we prove that the matrix \eqref{eq:D-even-matrix} is not diagonal in general (\S \ref{ssec:D-n=4}). 
\end{remark}

\subsection{Acknowledgments}

The author was supported by the NSF (grants DMS-2302520 and DMS-2441922), the Simons Foundation, and the Alfred P. Sloan Foundation. The author is indebted to Zhiwei Yun and Wei Zhang for a stimulating collaboration which led to the problem studied here, and Trieu Trinh for technical help. This work was done in collaboration with the Superhuman Reasoning team at Google DeepMind led by Thang Luong. 

\section{Declaration of AI Usage}\label{sec:origin-story}
The core mathematical content of this paper was entirely generated by an internal reasoning agent built upon Gemini Deep Think, codenamed {\Aletheia} at Google DeepMind \cite{aletheia}. The paper itself was (re)written by the human author, starting from \Aletheia's output. The ``Human--AI Interaction Card'' (see \cite[\S 5]{aletheia} for explanation of the concept) below this paragraph summarizes the interaction, and links to the raw model outputs. By comparing the paper to the original outputs, the reader can verify that \Aletheia's original output is already complete and correct in mathematical content. Thus the author's only contribution to this paper (aside from developing \Aletheia) was to fiddle with the exposition, and to write an Introduction. 

\begin{interactionlog}[]{https://github.com/google-deepmind/superhuman/tree/main/aletheia/F26}

    \human{Query eigenweights for Type A.}

    \ai{\Aletheia}{Response (fully correct).}

    \human{Proceed to query eigenweights for Type C.}

    \ai{\Aletheia}{Response (fully correct).}

    \human{Proceed to query eigenweights for Type D.}

    \ai{\Aletheia}{Response (fully correct).}

\end{interactionlog}

\section{Eigenweights for type A}

\subsection{Setup} 

Fix $n \geq 2$ and let $G = \GL_n$. Then we have $R \cong  \Q[\XX^*(T)] \cong \Q[x_1, \ldots, x_n]$, with $\deg x_i = 2$.  The Weyl group $W$ can be identified with the symmetric group $S_n$, and its natural action on $R$ is via permutation of the $x_i$. Let $X := \{x_1, \ldots, x_n\}$. The invariant ring $R^W = \rH^*(\BB G)$ is the ring of symmetric polynomials in $X$. Let $p_i(X) = \sum_{x_k \in X} x_k^i$ be the $i$th power sum in $X$. Then we have 
\[
R^W = \Q[p_1(X), \ldots, p_n(X)]
\]
and the $\{p_i(X)\}_{i=1}^n$ form a basis for $\VV = I/I^2$, where we recall that $I \subset R^W$ is the augmentation ideal.

\subsubsection{Minuscule coweight} Fix $1 \leq m <n$ and the minuscule coweight $\mu = (1^m, 0^{n-m})$ where $n>m$. The invariant ring  $\rH^*(\BB P_\mu) \cong \rH^*(\BB L_\mu) \cong R^{W_\mu}$ consists of polynomials symmetric separately in the first $m$ variables $A:= \{x_1, \ldots, x_m\}$ and the remaining $n-m$ variables $B := \{x_{m+1}, \ldots, x_n\}$.

The associated Levi subgroup is $L_\mu \cong \GL_m \times \GL_{n-m}$, whose Weyl group is $W_\mu = S_m \times S_{n-m}$. The flag variety $G/P_\mu$ identifies with $\Gr(m, n)$, and has dimension $D_\mu := m(n-m)$. We let $N := D_\mu+1$.

\subsubsection{Differential operator} The differential operator $\partial_\mu$ expressed in these terms is 
    $$ \partial_\mu = \sum_{i=1}^m \frac{\partial}{\partial x_i}. $$
    This operator carries $R^W$ to $R^{W_\mu}$.

\subsubsection{Casimir} We work with the normalization $\Omega = \frac{1}{2}\sum_{i=1}^n x_i^2$ of the Casimir element. Then we have 
\[
t_\mu = \partial_\mu (\Omega) = x_1+\cdots+x_m \in R^{W_\mu}.
\]
For a set $S$, let $p_i(S) = \sum_{s \in S} s^i$ be the $i$th power sum in $S$. Thus $t_\mu = p_1(A)$ and $\eta = p_1(A)^{N}$. 

\subsection{Proof of Theorem \ref{thm:A-main}}
We consider the operator $\ol {\nabla}^\eta_\mu$ on $\VV := I/I^2$ defined by the formula 
\[
\ol {\nabla}^\eta_\mu (f) = \int_{G/P_\mu} (p_1(A)^{N} \partial_\mu f) .
\]

Note that $\VV = I/I^2$ is 1-dimensional in each degree where it is non-zero. The power sums $p_1(X), \ldots, p_n(X)$ are a homogeneous basis for this grading, so we know a priori that they will be eigenvectors. We wish to find the associated eigenweights $\epsilon_k(\Omega, \mu)$ such that 
\[
\ol {\nabla}^\eta_\mu (p_k(X)) \equiv \epsilon_k(\Omega, \mu) p_k(X) \pmod{I^2}.
\]

For $k = 1, \ldots, n$, we have\footnote{even for $k=1$, since $p_0(A) = m$}
\begin{equation}
 \partial_\mu p_k(X) = \sum_{i=1}^m \frac{\partial}{\partial x_i} \sum_{j=1}^n x_j^k = k \sum_{i=1}^m x_i^{k-1} = k p_{k-1}(A). 
 \end{equation}

Next we consider 
\begin{equation}\label{eq:integrand}
\ov \nabla^{\eta}_\mu (p_k(X)) = k \int_{G/P_\mu} p_1(A)^{N} p_{k-1}(A).
\end{equation}
For a partition $\gamma = (\gamma_1, \ldots, \gamma_r)$, we define 
\begin{equation}\label{eq:partition-polynomial}
p_{\gamma} (A) := p_{\gamma_1}(A) p_{\gamma_2}(A) \ldots p_{\gamma_r}(A)
\end{equation}
In these terms, we can rewrite $p_1(A)^{N} p_{k-1}(A) = \Lambda_k p_{\nu_k}(A)$, where $\nu_k$ is the partition\footnote{for $k=1$, we interpret this as $\nu_1=(1^{N})$.} $(k-1, 1^{N})$ and 
\[
\Lambda_k := \begin{cases} m  & k=1, \\ 1 & k > 1. \end{cases}
\]

\subsubsection{Recollections on Schur polynomials}
We recall some definitions related to partitions and Schur polynomials. A reference is \cite[\S 7.10]{Stan2}. Let $\pi = (\pi_1 \geq  \pi_2 \geq  \ldots \geq \pi_l )$ be a partition. We write 
\[
\len(\pi) := l, \quad |\pi| := \pi_1 + \pi_2 + \ldots + \pi_l.
\]
The \emph{Young diagram} of $\pi$ has $\pi_1$ boxes in the first row, $\pi_2$ boxes in the second row, etc. A \emph{Semistandard Young Tableau} (SSYT) $T$ is a filling of the Young diagram with integers from $\{1, 2, \ldots, k\}$ such that entries increase weakly along rows and strictly down each column. A \emph{Standard Young Tableau} (SYT) is a filling of the Young diagram with the integers $\{1, \ldots, |\pi|\}$ such that each integer is used exactly once, and moreover entries increase \emph{strictly} along rows and columns. 

For a partition $\pi$, we denote by $\SSYT(\pi, k)$ the set of SSYT for $\pi$ with entries in $\{1, \ldots, k\}$. To $T \in \SSYT(\pi, k)$, we define  
\[
x^T := \prod_{i=1}^k x_i^{n_i}, \quad \text{$n_i = $ number of times that $i$ appears in $T$}.
\]
The \emph{Schur polynomial} in $\{x_1, \ldots, x_k\}$ associated to a partition $\pi$ is 
\[
s_\pi(\{x_1, \ldots, x_k\}) = \sum_{T \in \SSYT(\pi, k)} x^T.
\]

\subsubsection{Integration}
The pushforward map $\BB P_\mu \rightarrow \BB G$ induces on cohomology a map 
\[
\int \co  R^{W_\mu} \to R^W.
\]

\begin{lemma}\label{lem:A-integral}
Let $\Pi:=((n-m)^m)$ be the $m \times (n-m)$ rectangular partition. Then we have 
$$ 
\int_{G/P_\mu} s_{\lambda}(A) = (-1)^{D_\mu} \begin{cases} s_{\lambda-\Pi}(X) & \lambda \supseteq \Pi, \\ 0 & \text{otherwise}. 
\end{cases}
$$
\end{lemma}

This Lemma resembles results in \cite[\S 14]{Ful98}, but we did not find an exact reference in the literature. 

\begin{proof} 
The Vandermonde determinant for a set of variables $Y = \{y_1, \ldots, y_r\}$ is denoted by
\[
\Delta(Y) = \prod_{1 \leq i < j \leq r} (y_i - y_j).
\]

Let $\frR_{\mu} \in \rH_G^*(G/P_\mu) \cong R^{W_{\mu}}$ be the $G$-equivariant top Chern class of the tangent bundle of $G/P_{\mu}$. Explicitly, we have
\begin{equation}\label{eq:equivariant-top-Chern}
\frR_{\mu} :=\prod_{\j{\a,\mu} < 0}\a\in R^{W_{\mu}}
\end{equation}
where the product runs over roots of $G$. In this case, we have explicitly 
\[
\frR_{\mu} = (-1)^{D_\mu} \prod_{i=1}^m \prod_{j=m+1}^n (x_i-x_j)
\]

Recall that we defined $X:=\{x_1, \ldots, x_n\}$ and $A :=\{x_1, \ldots, x_m\}$. Write $B:=\{x_{m+1}, \ldots, x_n\}$ so that $X = A \sqcup B$. Then we have the relation 
\begin{equation}\label{eq:vandermonde-eq1}
\Delta(X) = \Delta(A) \cdot \Delta(B) \cdot \prod_{i=1}^m \prod_{j=m+1}^n (x_i - x_j) =   (-1)^{D_\mu}  \Delta(A) \Delta(B) \mathfrak{R}_\mu.
\end{equation}

For a set of variables $Y = (y_1, \ldots, y_r)$ and a partition $\gamma = (\gamma_1, \ldots, \gamma_r)$, write $Y^\gamma := y_1^{\gamma_1} \ldots y_r^{\gamma_r}$. Next we invoke the Weyl Character Formula for Schur polynomials, also known as \emph{Cauchy's bialternant formula} \cite[Theorem 7.15.1]{Stan2}, which says that for $\delta_r = (r-1, r-2, \ldots, 0)$, 
\begin{equation}\label{eq:bialternant}
s_\nu(Y) = \frac{a_{\nu + \delta_r}(Y)}{\Delta(Y)}, \quad \text{where }a_\gamma(Y) = \sum_{w \in S_r} \operatorname{sgn}(w) w(Y^\gamma). 
\end{equation}
Note that we may view $\Delta(B) = a_{\delta_{n-m}}(B)$. 

By \cite[Lemma 5.4.4]{FYZ4} (which is essentially Atiyah--Bott localization for partial flag varieties) for any $f\in R^{W_{\mu}}$, we have
\begin{equation}\label{Av w f/R}
\int_{G/P_{\mu}}f=\sum_{w\in W/W_{\mu}}w(f/\frR_{\mu}).
\end{equation}
Here the sum is over a set of representatives of $W/W_{\mu}$, which is well-defined since $f/\frR_{\mu}$ is $W_{\mu}$-invariant.\footnote{The element $f/\frR_{\mu}$ is understood in the fraction field of $R$, but the identity implicitly mandates that the sum above will lie in $R^{W}$.} Substituting $f = s_\lambda(A) = a_{\lambda+\delta_m}(A)/\Delta(A)$ into \eqref{Av w f/R}, we obtain
\begin{align}\label{eq:A-schur-1}
\int_{G/P_\mu} s_\lambda(A) & =   \sum_{w \in W/W_\mu} w \left( \frac{s_\lambda(A)}{\mathfrak{R}_\mu} \right) =   \sum_{w \in W/W_\mu} w \left( \frac{a_{\lambda+\delta_m}(A)}{\Delta(A) \mathfrak{R}_\mu} \right) \nonumber \\
& \stackrel{\eqref{eq:vandermonde-eq1}}= (-1)^{D_\mu} \sum_{w \in W/W_\mu} w \left( \frac{a_{\lambda+\delta_m}(A) \Delta(B)}{\Delta(X)} \right).
\end{align}
Since $w(\Delta(X)) = \operatorname{sgn}(w) \Delta(X)$, we can pull the denominator out of the summation with a sign change, thus rewriting \eqref{eq:A-schur-1} as 
\begin{equation}\label{eq:schur-polynomial}
\int_{G/P_\mu} s_\lambda(A) = \frac{(-1)^{D_\mu}}{\Delta(X)} \sum_{w \in W/W_\mu} \operatorname{sgn}(w) w \left( a_{\lambda+\delta_m}(A) \Delta(B) \right).
\end{equation}

Note that $a_{\lambda+\delta_m}(A) \Delta(B)  = a_{\lambda+\delta_m}(A) a_{\delta_{n-m}}(B)  $ is already antisymmetrized with respect to the subgroup $W_\mu = S_m \times S_{n-m}$, meaning that
\[
a_{\lambda+\delta_m}(A) a_{\delta_{n-m}}(B) = \left( \sum_{\sigma \in S_m} \operatorname{sgn}(\sigma) \sigma(A^{\lambda+\delta_m}) \right) \left( \sum_{\tau \in S_{n-m}} \operatorname{sgn}(\tau) \tau(B^{\delta_{n-m}}) \right).
\]
Therefore, further symmetrizing over $W/W_\mu$ can be rewritten as an antisymmetrization over all of $W$: 
\begin{equation}\label{eq:numerator-weyl-sum}
\sum_{w \in W/W_\mu} \operatorname{sgn}(w) w \left( a_{\lambda+\delta_m}(A) a_{\delta_{n-m}}(B) \right) = \sum_{\rho \in S_n} \operatorname{sgn}(\rho) \rho \left( A^{\lambda+\delta_m} B^{\delta_{n-m}} \right).
\end{equation}
By definition, we have 
\[
A^{\lambda+\delta_m} B^{\delta_{n-m}} = 
x_1^{\lambda_1+m-1} x_2^{\lambda_2+m-2} \cdots x_m^{\lambda_m} \cdot x_{m+1}^{n-m-1} x_{m+2}^{n-m-2} \cdots x_n^0.
\]
Let $\gamma$ be the partition $(\lambda_1+m-1, \ldots, \lambda_m, n-m-1, \ldots, 0)$. Then \eqref{eq:numerator-weyl-sum} can be written as $a_\gamma(X)$, and inserting this into \eqref{eq:schur-polynomial} gives 
\[
\int_{G/P_\mu} s_\lambda(A) =(-1)^{D_\mu}  \frac{a_\gamma(X)}{\Delta(X)}.
\]

It remains to show that 
\[
\frac{a_\gamma(X)}{\Delta(X)} = \begin{cases} s_{\lambda-\Pi}(X) & \lambda \supseteq \Pi, \\ 0 & \text{otherwise.}\end{cases}
\]
By its definition as an antisymmetrization, it is clear that $a_\gamma(X)$ vanishes if the components of $\gamma$ are not distinct. The first $m$ components of $\gamma = (\lambda_1+m-1, \ldots, \lambda_m, n-m-1, \ldots, 0)$ are strictly decreasing (since $\lambda$ is a partition), and the last $n-m$ components are clearly strictly decreasing. Hence $\gamma$ is strictly decreasing if and only if the $m$-th component is strictly greater than the $(m+1)$-st component, i.e., $\lambda_m \geq n-m$. If this last inequality is not satisfied, then $0 \leq \lambda_m < (n-m)$ so $\lambda_m \in \{0, \ldots, n-m-1\}$. Since this value also appears among the last $n-m$ components of $\gamma$, we conclude that $\gamma$ has repeated components, so the alternant $a_\gamma(X)=0$.

Given that $\len(\lambda)\le m$, this condition is equivalent to $\lambda$ containing the rectangular partition $\Pi=((n-m)^m)$. This shows that $\int_{G/P_\mu} s_{\lambda}(A) = 0$ if $\lambda \not\supseteq \Pi$. 

Finally, suppose $\lambda \supseteq \Pi$. Then we have
\[
\gamma - \delta_n   = (\lambda_1 +m-n, \lambda_2+m-n, \ldots, \lambda_m+m-n, 0, 0, \ldots, 0) 
\]
which is precisely the partition $\lambda -  \Pi$. This completes the proof.

\end{proof}

\subsubsection{Frobenius character formula} 
For a partition $\nu$ of an integer $M$, let $p_\nu$ be as in \eqref{eq:partition-polynomial}. The irreducible representations of $S_{M}$ are indexed by partitions of $M$. For a partition $\pi$ of $M$, let $\chi^\pi$ be the corresponding character. For a partition $\nu$ of $M$, we denote by $\chi^\pi(\nu)$ the value of $\chi^\pi$ on the conjugacy class of $S_M$ corresponding to $\nu$. We have the \emph{Frobenius character formula} \cite[p.49, \S 4.1]{FH91}, (see also \cite[(7.8)]{Mac95} for a reference with a proof)
\begin{equation}\label{eq:Frob-char-formula}
p_\nu(A) = \sum_{|\pi|= M} \chi^\pi(\nu) s_\pi(A),
\end{equation}
where the sum runs over partitions $\pi$ of $M$. Note that $s_\pi(A)$ vanishes automatically if $\pi$ has more than $m$ parts, so the sum could be restricted to partitions with at most $m$ parts; we shall use this below. Abbreviate $\nu_k := (k-1,1^{N})$. Inserting \eqref{eq:Frob-char-formula} into \eqref{eq:integrand} with $M = N+k-1$, we find that 
\begin{align*}
 \ol {\nabla}^\eta_\mu(p_k(X)) & =  k   \Lambda_k  \int_{G/P_\mu} p_{\nu_k}(A) =  k \Lambda_k\sum_{|\pi| = k +m(n-m)} \int_{G/P_\mu}  \chi^\pi(\nu_k) s_\pi(A).
 \end{align*}
Then applying Lemma \ref{lem:A-integral}, we obtain
\begin{equation}\label{eq:a-proof-1}
 \ol {\nabla}^\eta_\mu(p_k(X))   =  (-1)^{D_\mu}  k \Lambda_k\sum_{\substack{|\pi| = k +m(n-m) \\ \len(\pi) \leq m \\ \pi \supseteq \Pi}} \chi^{\pi}(\nu_k) s_{\pi-\Pi}(X) = (-1)^{D_\mu}    k \Lambda_k \sum_{\substack{|\pi|= k \\ \len(\pi)\le m}} \chi^{\pi+\Pi}(\nu_k) s_{\pi}(X).
\end{equation}

\begin{lemma}\label{lem:schur-to-power}
    In $I/I^2$ we have 
    $$ s_\pi(X) \equiv \frac{\chi^\pi((k))}{k} p_k(X) \pmod{I^2}. $$
\end{lemma}

\begin{proof}
 An alternate form of the Frobenius character formula \cite[Proof of (7.6)(i)]{Mac95} says 
    \begin{equation}\label{eq:Frob-char-formula-2}
    s_\pi(X) = \sum_{|\nu|=k} z_\nu^{-1} \chi^\pi(\nu)p_\nu(X) 
    \end{equation}
    where the sum runs over all partitions $\nu$ of $k = |\pi|$, and $z_\nu$ is the size of the centralizer of a permutation with cycle type $\nu$. Explicitly, if $\nu = (1^{m_1}, 2^{m_2} \ldots)$ then $z_\nu = \prod_i i^{m_i} m_i!$.  All the terms lie in $I^2$ except the one indexed by $\nu = (k)$, and we see that $z_{\nu} = k$ in this case, which gives the result. 
\end{proof}

Inserting Lemma \ref{lem:schur-to-power} into \eqref{eq:a-proof-1}, we obtain that  
\begin{equation}
\ol {\nabla}^\eta_\mu (p_k(X)) \equiv  (-1)^{D_\mu} \Lambda_k \sum_{\substack{\pi \vdash k \\ \len(\pi)\le m}} \chi^{\pi + \Pi}(\nu_k) \chi^\pi((k)) p_k (X) \pmod{I^2}.
\end{equation}
This gives a formula for the eigenweight associated to $p_k$, 
\begin{equation}\label{eq:a-proof-2}
\epsilon_k = (-1)^{D_\mu}  \Lambda_k \sum_{\substack{|\pi| =  k \\ \len(\pi)\le m}} \chi^{\pi + \Pi}(\nu_k) \chi^\pi((k)). 
\end{equation}

\subsubsection{Application of the Murnaghan--Nakayama Rule}\label{ssec:MN-rule}
Next we will simplify $\chi^\pi((k))$ using the Murnaghan--Nakayama Rule \cite[\S 7.17]{Stan2}, which gives a recursive algorithm for computing $\chi^\pi(\nu)$ for general partitions $\pi$ and $\nu$. For the convenience of the reader, we recall its formulation.

Given a Young diagram $\cD$, a \emph{border strip} (also called \emph{rim hook}) $\beta$ is a connected path of boxes along the outer rim of $\cD$ that satisfies the conditions:
\begin{itemize}
    \item It is ``connected," meaning all its boxes touch at an edge.
    \item It does \emph{not} contain any $2 \times 2$ square of boxes.
    \item Removing the border strip leaves a valid Young diagram denoted $\pi \setminus \beta$.
\end{itemize}

The \emph{height} of a border strip is defined as
\[
\height(\beta) = (\text{Number of rows the strip occupies}) - 1.
\]

Let $\nu = (\nu_1, \nu_2, \ldots, \nu_k)$. For each border strip $\beta$, let $\pi \setminus \beta$ be the partition obtained by removing the border strip $\beta$ from the Young diagram of $\pi$. Let $\nu' := (\nu_2, \ldots)$ be the partition obtained by removing $\nu_1$ from $\nu$. The Murnaghan--Nakayama rule says that 
    \[
    \chi^\pi(\nu) = \sum_{|\beta|=\nu_1} (-1)^{\height(\beta)} \chi^{\pi \setminus \beta}(\nu')
    \]
    where the sum runs over all border strips $\beta$ of length $\nu_1$. The base cases are determined by $\chi^{\emptyset}(\emptyset) = 1$, and if at any step a border strip of the required size does not exist, then that entire branch of the calculation contributes $0$.


\begin{example}\label{ex:MN-rule-2}
We compute $\chi^{(4,3,2,1)}((3, 1^7))$ using the Murnaghan--Nakayama rule. There are 3 border strips in the Young diagram for $(4,3,2,1)$. 
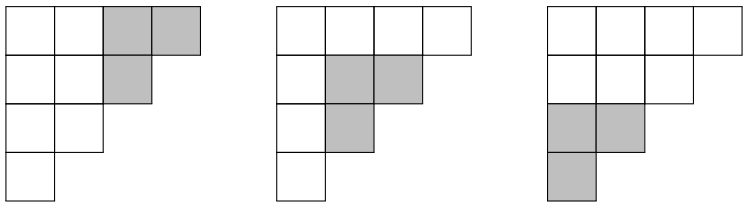
\[
\ytableausetup{boxsize=1.8em}
\begin{ytableau}
  ~ & ~ & *(lightgray) ~ & *(lightgray) ~ \\
  ~ & ~ & *(lightgray) ~ \\
  ~ & ~ \\
  ~
\end{ytableau} \hspace{1cm}
\ytableausetup{boxsize=1.8em}
\begin{ytableau}
  ~ & ~ & ~ & ~ \\
  ~ & *(lightgray) ~ & *(lightgray) ~ \\
  ~ & *(lightgray) ~ \\
   ~
\end{ytableau} \hspace{1cm}
\ytableausetup{boxsize=1.8em}
\begin{ytableau}
  ~ & ~ & ~ & ~ \\
  ~ & ~ & ~ \\
  *(lightgray) ~ & *(lightgray) ~ \\
  *(lightgray) ~
\end{ytableau}
\]
Hence we have 
\[
\chi^{(4,3,2,1)}((3, 1^7))  = (-1)^1 ( \chi^{(2,2,2,1)}(1^7) + \chi^{(4,1,1,1)}(1^7) + \chi^{(4,3)}(1^7)).
\]
Each of the inner summands can be computed by the Hook Length Formula. 
\[
\ytableausetup{boxsize=1.8em} 
\begin{ytableau}
  5 & 3 \\
  4 & 2 \\
  3 & 1  \\
  1
\end{ytableau} \hspace{2cm}
\ytableausetup{boxsize=1.8em} 
\begin{ytableau}
  7 & 3 & 2 & 1 \\
  3 \\
  2 \\
  1
\end{ytableau}\hspace{1cm}
\ytableausetup{boxsize=1.8em} 
\begin{ytableau}
  5 & 4 & 3 & 1 \\
  3 & 2 & 1
\end{ytableau}
\]
This gives $\chi^{(2,2,2,1)}(1^7) = \chi^{(4,3)}(1^7) = 14$ and $\chi^{(4,1,1,1)}(1^7) = 20$. Therefore, we conclude that
\[
\chi^{(4,3,2,1)}(3, 1^7) = (-1)^1 (14+20+14) = \boxed{-48.}
\]
\end{example}

\begin{lemma}\label{lem:char-val-MN}
Suppose $\pi$ is a partition of $k$. Then $\chi^{\pi}((k))$ vanishes unless $\pi =\pi_j(k) := (k-j, 1^j)$ for $j = 0, 1, \ldots, k-1$, in which case we have
\[
\chi^{(k-j, 1^j)}((k)) =(-1)^j.
\]
\end{lemma}

\begin{proof}
By the Murnaghan--Nakayama rule, the character $\chi^\pi((k))$ is non-zero only if $\pi$ is a border strip for itself, which occurs if and only if $\pi$ is a hook. The partitions of $k$ which are hooks are $\pi_j(k) :=(k-j, 1^j)$ for $j=0, \ldots, k-1$. Since $\pi_j$ has height $j$, the Murnaghan--Nakayama rule implies that $\chi^{\pi_j}((k)) = (-1)^j$.
    \end{proof}

We apply Lemma \ref{lem:char-val-MN} to simplify \eqref{eq:a-proof-2}. Note that the constraint $\len(\pi)=j+1\le m$ implies $j\le m-1$. Putting these observations into \eqref{eq:a-proof-2}, we arrive at the desired formula
\[
\epsilon_k= (-1)^{D_\mu}  \Lambda_k \sum_{j=0}^{\min(k-1, m-1)} (-1)^j \chi^{\pi_j(k) + \Pi}(\nu_k).
\]    
This finally completes the proof of Theorem \ref{thm:A-main}. \qed 

\subsection{Examples}\label{ssec:a-special-cases}

We explain how Theorem \ref{thm:A-main} specializes to the results of \cite{FYZ4} for $m=1$ and $m=2$.

\subsubsection{The case $m=1$}
Fix $n \geq 2$ and $m=1$. Applying Theorem \ref{thm:A-main}, the sum only has a contribution from $j=0$, in which case, $\pi_j = (k)$, and $\pi_j(k) + (n-m)^m = (n+k-1)$. Thus we have
\[
\epsilon_k = (-1)^{n-1} \chi^{(n+k-1)} ((k-1, 1^n))
\]
which by the Murnaghan--Nakayama rule equals $(-1)^{n-1}  \chi^{(n)}(1^n) = (-1)^{n-1} $. This recovers \cite[Proposition 4.4.1]{FYZ4}. 

\subsubsection{The case $m=2$}
Fix $n \geq 3$ and $m=2$. Assume $k \geq 2$; the case $k=1$ is similar but simpler, with some notation changes needed. We have $\Pi = (n-2,n-2)$ and $\nu_k = (k-1, 1^{2(n-2)+1})$. Note that $(-1)^{N-1} = (-1)^{2(n-2)} = 1$ in this case. 

Applying Theorem \ref{thm:A-main}, we see that only $j=0$ and $j=1$ contribute to the sum. Thus we have
    \begin{equation}\label{eq:m=2eps}
    \epsilon_k = \chi^{(k+n-2, n-2)}((k-1, 1^{2(n-2)+1}))  -  \chi^{(k+n-3, n-1)} ((k-1, 1^{2(n-2)+1})).
    \end{equation}
    We evaluate these using the Murnaghan-Nakayama rule. 
    \begin{itemize}
    \item First we calculate $\chi^{(k+n-2, n-2)}((k-1, 1^{2(n-2)+1})) $. The only possible border strips are horizontal (height $0$) of size $k-1$ in the first row, or (if $k<n$) the second row, giving 
    \[
    \chi^{(k+n-2, n-2)}((k-1, 1^{2(n-2)+1}))  = \chi^{(n-1,n-2)}(\Id) + \chi^{(k+n-2, n-k-1)}(\Id).
    \]

\item Next we calculate the term $\chi^{(k+n-3, n-1)} ((k-1, 1^{2(n-2)+1})) $. The only border strip is a horizontal one (height 0) in the second row. Thus we have 
\[
\chi^{(k+n-3, n-1)} ((k-1, 1^{2(n-2)+1})) = \chi^{(k+n-3, n-k)}(\Id).
\]
\end{itemize}
Plugging these back into \eqref{eq:m=2eps}, we arrive at 
\begin{equation}\label{eq:m=2eps-2}
\epsilon_k = \chi^{(n-1,n-2)}(\Id) + \chi^{(k+n-2, n-k-1)}(\Id) - \chi^{(k+n-3, n-k)}(\Id).
\end{equation}

For a partition of the form $\lambda = (\lambda_1, \lambda_2)$, the Hook Length Formula \eqref{eq:hook-length-formula} specializes to
$$
\dim(\lambda_1, \lambda_2) = \frac{N! (\lambda_1 - \lambda_2 + 1)}{(\lambda_1+1)! \lambda_2!}
$$
In particular, we have: 
\begin{align*}
\chi^{(n-1,n-2)}(\Id)  &= \frac{(2n-3)!  2}{n! (n-2)!} \\
\chi^{(k+n-2, n-k-1)}(\Id) & =  \frac{(2n-3)! (2k)}{(k+n-1)! (n-k-1)!}  \\ 
\chi^{(k+n-3, n-k)}(\Id) &= \frac{(2n-3)! (2k-2)}{(k+n-2)! (n-k)!}
\end{align*}
Substituting these into \eqref{eq:m=2eps-2}, we obtain
\begin{equation}
\epsilon_k = \frac{(2n-3)!  2}{n! (n-2)!} + \frac{(2n-3)! (2k)}{(k+n-1)! (n-k-1)!} -  \frac{(2n-3)! (2k-2)}{(k+n-2)! (n-k)!}.
\end{equation}
After some tedious simplification, this agrees with \eqref{eq:FYZ-epsilon-2}, thus recovering \cite[Proposition 5.4.1]{FYZ4}. 

\section{Eigenweights for type C}

\subsection{Setup} 

Fix $n \geq 1$ and let $G=\PSp_{2n}$. Then we have $R = \rH^*(\BB T; \Q) = \Q[x_1, \ldots, x_n]$, with $\deg x_i = 2$. The Weyl group $W$ (Type $C_n$) consists of signed permutations on the $x_i$. Let $X := \{x_1, \ldots, x_n\}$ and $X^{(2)}:=\{x_1^2, \ldots, x_n^2\}$. The invariant ring is $R^W = \Q[p_1^{(2)}, \ldots, p_n^{(2)}]$ where 
\[
\underbrace{p_k^{(2)}}_{\deg 4k} := p_k(X^{(2)}) \quad \text{is the $k$-th power sum in $X^{(2)}$}.
\]
A basis for $\VV := I/I^2$ is $\{p_k^{(2)}\}_{k=1}^n$.

\subsubsection{Minuscule coweights}
The unique minuscule coweight of $G$ is $\mu = \frac{1}{2}(1, \ldots, 1)$. The Levi subgroup $L_\mu$ is of type $A_{n-1}$, and its Weyl group is $W_\mu=S_n$. The invariant ring $R^{W_\mu}$ is the ring of symmetric polynomials in $X$. The dimension of $G/P_\mu$ is $D_\mu := \binom{n+1}{2}$. Let $N := D_\mu+1$. 

\subsubsection{Differential operator}
The differential operator $\partial_\mu$ expressed in these terms is 
\[
\partial_\mu = \frac{1}{2} \sum_{i=1}^n \pderiv{}{x_i}.
\]
This operator carries $R^W$ to $R^{W_\mu}$.

\subsubsection{Casimir} We take the Casimir element $\Omega = \frac{1}{2}\sum_{i=1}^n x_i^2$. Then we have 
\[
t_\mu = \partial_\mu(\Omega) = \frac{1}{2}(x_1 + \ldots + x_n) = \frac{1}{2}p_1(X),
\] 
where again $p_i$ refers to the $i$th power sum. Thus $\eta = 2^{-N} p_1(X)^{N}$. 

\subsection{Proof of Theorem \ref{thm:C-main}}
We consider the operator $\ol {\nabla}^\eta_\mu$ on $\VV := I/I^2$ defined by the formula 
\[
\ol {\nabla}^\eta_\mu (f) = 2^{-N}  \int_{G/P_\mu} (p_1(X)^{N} \partial_\mu f) .
\]

Note that $\VV = I/I^2$ is 1-dimensional in each degree where it is non-zero. Recalling that $X^{(2)} := \{x_1^2, \ldots, x_n^2\}$, the power sums $\{p_i^{(2)} = p_i(X^{(2)})\}_{i=1}^n$ form a homogeneous degree basis for $\VV$, so they are eigenvectors for $\ol {\nabla}^\eta_\mu$. We wish to find the associated eigenweights $\epsilon_k(\Omega, \mu)$ such that 
\[
\ol {\nabla}^\eta_\mu (p_k^{(2)}) \equiv \epsilon_k(\Omega, \mu) p_k^{(2)} \pmod{I^2}.
\]

\subsubsection{Integration}

Let $\rho_n = (n, n-1, \ldots, 1)$. 

\begin{lemma}\label{lem:C-integration}
Let $\lambda$ be a partition with at most $n$ parts. The Gysin map acts on Schur polynomials as follows:
$$
\int_{G/P_\mu} s_\lambda(X) = \begin{cases} (-1)^{D_\mu} s_\pi(X^{(2)}) & \text{if } \lambda = 2\pi+\rho_n \text{ for some partition } \pi, \\ 0 & \text{otherwise.} \end{cases}
$$
\end{lemma}

\begin{proof} We will apply \cite[Lemma 5.4.4]{FYZ4} to $s_\lambda(X)$. In preparation, we analyze $s_\lambda(X)/\frR_\mu$. By the bialternant formula \eqref{eq:bialternant}, we have $s_\lambda(X) = a_{\lambda + \delta_n}(X)/\Delta(X)$. Writing $R^+(X):=\prod_{i<j}(x_i+x_j)$, we have
\[
\frR_\mu = (-1)^{D_\mu} \prod_{i=1}^n (2x_i) \prod_{i<j} (x_i+x_j) = (-1)^{D_\mu} 2^n e_n(X) R^+(X).
\]
Observe the identity $\Delta(X^{(2)}) = \Delta(X)R^+(X)$. Hence we may rewrite
$$ \frac{s_\lambda(X)}{\frR_\mu} = \frac{(-1)^{D_\mu} a_{\lambda + \delta_n}(X)}{2^n e_n(X) \Delta(X) R^+(X)} = \frac{(-1)^{D_\mu} a_{\lambda + \delta_n}(X)}{2^n e_n(X) \Delta(X^{(2)})}. $$

The quotient $W/W_\mu$ is the group $K \cong \{\pm 1\}^n$ of sign changes on $X$. Then $\mathcal{P}_K = \frac{1}{2^n}\sum_{\sigma \in K} \sigma(\cdot)$ can be interpreted as the projection operator onto the $K$-invariant subspace $R^K=\Q[X^{(2)}]$. Note that $\Delta(X^{(2)})$ is $K$-invariant. Applying these observations, \cite[Lemma 5.4.4]{FYZ4} says that
\begin{equation}\label{eq:C-proof-eq-intermediate}
\int_{G/P_\mu} s_\lambda(X) = \frac{(-1)^{D_\mu}}{\Delta(X^{(2)})} \mathcal{P}_K\left(\frac{a_{\lambda + \delta_n}(X)}{e_n(X)}\right).
\end{equation}
Let $h_\gamma(X) := a_\gamma(X)/e_n(X) = \sum_{w \in S_n} \mathrm{sgn}(w) X^{w(\gamma)-(1^n)}$.
The operator $\mathcal{P}_K$ extracts the monomials with purely even exponents. A term $X^\beta$ in $h_\gamma(X)$ has exponent $\beta=w(\gamma)-(1^n)$. This is even if and only if $w(\gamma)$ is purely odd, which is equivalent to $\gamma$ itself being purely odd. Hence we see that if $\lambda + \delta_n$ is not purely odd, then $\int_{G/P_\mu} s_\lambda (X)=0$.

If $\lambda+\delta_n$ is purely odd, we may write $\lambda+\delta_n=2\kappa+(1^n)$ for another partition $\kappa$. Since $\lambda+\delta_n$ is strictly decreasing, $\kappa$ is also strictly decreasing. Since $w(\lambda+\delta_n)-(1^n)=2w(\kappa)$ is even, $h_{\lambda+\delta_n}(X)$ is $K$-invariant, and $\mathcal{P}_K(h_{\lambda+\delta_n}(X))=h_{\lambda+\delta_n}(X)$. Hence
$$ h_{\lambda+\delta_n}(X) = \sum_{w \in S_n} \mathrm{sgn}(w) X^{2w(\kappa)} = \sum_{w \in S_n} \mathrm{sgn}(w) (X^{(2)})^{w(\kappa)} = a_\kappa(X^{(2)}). $$
Putting this into \eqref{eq:C-proof-eq-intermediate}, we find that
\[
\int_{G/P_\mu} s_\lambda(X) = (-1)^{D_\mu} \frac{a_\kappa(X^{(2)})}{\Delta(X^{(2)})}
\]
where $\lambda + \delta_n = 2 \kappa  + (1^n)$. Since $\kappa$ is strictly decreasing, we can write $\kappa=\pi+\delta_n$ for a unique partition $\pi$. By another application of \eqref{eq:bialternant}, we recognize $a_\kappa(X^{(2)})/\Delta(X^{(2)})  = s_\pi(X^{(2)})$ where $\lambda = 2 \pi + \rho_n$. This completes the proof. 
\end{proof}

\subsubsection{Action on $p_k^{(2)}$}
We calculate the action of $\ol{\nabla}^\eta_\mu$ on $p_k^{(2)}$ for $k=1, \ldots, n$. Note that 
$$ \partial_\mu p_k^{(2)} = \frac{1}{2} \sum_{j=1}^n \frac{\partial}{\partial x_j} \sum_{i=1}^n x_i^{2k} = k \sum_{j=1}^n x_j^{2k-1} = k p_{2k-1}(X). $$
Hence we have
$$ \ol{\nabla}^\eta_\mu(p_k^{(2)}) =  2^{-N} k \int_{G/P_\mu} p_1(X)^N p_{2k-1}(X) =  2^{-N} k  \int_{G/P_\mu} p_{\nu_k}(X), $$
where $\nu_k = (2k-1, 1^N)$. 

Applying the Frobenius character formula \eqref{eq:Frob-char-formula}, we see that 
$$ 
\ol{\nabla}^\eta_\mu(p_k^{(2)}) = 2^{-N} k  \sum_{|\lambda|=2k-1+N} \chi^\lambda(\nu_k) \int_{G/P_\mu} s_\lambda(X). 
$$
By Lemma \ref{lem:C-integration}, the sum is concentrated on $\lambda$ such that $\lambda=2\pi+\rho_n$, which implies that $|\pi|=k$. So we have 
\begin{equation}\label{eq:C-eigenweight-eq1}
\ol{\nabla}^\eta_\mu(p_k^{(2)}) = (-1)^{D_\mu} 2^{-N} k  \sum_{|\pi|=k} \chi^{2\pi+\rho_n}(\nu_k) s_\pi(X^{(2)}).
\end{equation}

We project this expression onto $\VV=I/I^2$. By \eqref{eq:Frob-char-formula-2}, when $|\pi|=k>0$ we have
$$ s_\pi(X^{(2)}) \equiv \frac{\chi^\pi((k))}{k} p_k^{(2)} \pmod{I^2}. $$
Substituting this back into \eqref{eq:C-eigenweight-eq1} gives
\begin{equation}\label{eq:C-eigenweight-eq2}
\ol{\nabla}^\eta_\mu(p_k^{(2)}) \equiv (-1)^{D_\mu} 2^{-N} \left(  \sum_{|\pi|=k} \chi^\pi((k)) \chi^{2\pi+\rho_n}(\nu_k) \right) p_k^{(2)} \pmod{I^2}.
\end{equation}
As we saw in Lemma \ref{lem:char-val-MN}, $\chi^\pi((k))$ is non-zero only if $\pi = \pi_j(k) = (k-j, 1^j)$ for $j=0, \ldots, k-1$, in which case we have $\chi^{\pi_j(k)}((k)) = (-1)^j$. Substituting this into \eqref{eq:C-eigenweight-eq2}, we obtain
$$ \epsilon_k(\Omega, \mu) = (-1)^{D_\mu} 2^{-N} \sum_{j=0}^{k-1} (-1)^j \chi^{2\pi_j(k)+\rho_n}(\nu_k). $$
This completes the proof of Theorem \ref{thm:C-main}. \qed 

\begin{example}
For $n=3$, Theorem \ref{thm:C-main} gives
\[
(\epsilon_1, \epsilon_2, \epsilon_3) = 2^{-4}\cdot (8,5,8)
\]
 and for $n=4$, Theorem \ref{thm:C-main} gives
 \[
( \epsilon_1, \epsilon_2, \epsilon_3, \epsilon_4) = 2^{-4}\cdot (44, 19, 28, 41).
 \]
\end{example}

\section{Eigenweights for type D}

\subsection{Setup}

Let $G=\PSO_{2n}$ with $n\geq 2$. Then $R=\Q[x_1, \ldots, x_n]$ with $\deg(x_i)=2$. Let $X := \{x_1, \ldots, x_n\}$ and $X^{(2)} := \{x_1^2, \ldots, x_n^2\}$. The Weyl group $W=W(D_n)$ consists of even signed permutations of $X$. The invariant ring is $R^W = \Q[p_1^{(2)}, \ldots, p_{n-1}^{(2)}, \Pf]$, where 
\[
\underbrace{p_k^{(2)}}_{\deg 4k} := p_k(X^{(2)}) = \sum_{i=1}^{n} x_i^{2k}   \quad \text{ and } \quad \underbrace{\Pf}_{\deg 2n} := \prod_{i=1}^n x_i. 
\]
The first augmentation quotient $\VV = I/I^2$ has basis $p_1^{(2)}, \ldots, \Pf$. 

\subsubsection{Minuscule coweights} In terms of the usual presentation of the $D_n$ root system $\Phi = \{ \pm x_i  \pm x_j | i \neq j\}$, the minuscule coweights are:
\begin{itemize}
\item $\omega_1^\vee = (1, 0, \ldots, 0)$, corresponding to the vector representation of $\SO_{2n}$. 
\item $\omega_{n}^\vee = \frac{1}{2}(1, 1, \ldots, 1)$ and $\omega_{n-1}^\vee = \frac{1}{2}(1, 1, \ldots, 1, -1)$. These correspond to the two spinor representations of $\Spin_{2n}$. They are exchanged by the order 2 outer automorphism coming from the flip of the Dynkin diagram.  
\end{itemize} 
The eigenweights for $\omega_1^\vee$ were computed in \cite{FYZ4}. Here we will undertake the computation of eigenweights of the spinor coweights, which is substantially more difficult. Since these are exchanged by an automorphism, their associated eigenweights will coincide, so it suffices to focus on $\mu := \omega_n^\vee$.

The Levi subgroup $L_\mu$ is of type $A_{n-1}$, and its Weyl group is $W_\mu=S_n$. The invariant ring $R^{W_\mu}$ is the ring of symmetric polynomials in $X$. The dimension of $G/P_\mu$ is $D_\mu := \binom{n}{2}$. Let $N := D_\mu+1$. 

\subsubsection{Differential operator} The differential operator $\partial_\mu$ expressed in these terms is 
\[
\partial_\mu = \frac{1}{2} \sum_{i=1}^n \pderiv{}{x_i}.
\]
This operator carries $R^W$ to $R^{W_\mu}$.

\subsubsection{Casimir} We take the Casimir element $\Omega = \frac{1}{2}\sum_{i=1}^n x_i^2$. Then we have 
\[
t_\mu = \partial_\mu(\Omega) = \frac{1}{2}(x_1 + \ldots + x_n) =  \frac{1}{2}p_1(X).
\] 
Thus $\eta = 2^{-N} p_1(X)^{N}$. 

\subsection{Proof of Theorem \ref{thm:D-main}}
We consider the operator $\ol {\nabla}^\eta_\mu$ on $\VV := I/I^2$ defined by the formula 
\[
\ol {\nabla}^\eta_\mu (f) = 2^{-N}  \int_{G/P_\mu} (p_1(X)^{N} \partial_\mu f) .
\]

Note that $\VV = I/I^2$ is 1-dimensional in each degree where it is non-zero, \emph{unless} $n=2m$ is even, in which case the $2n=4m$-graded piece $\VV_{2n}$ is spanned by $p_m^{(2)}$ and $\Pf$. Hence if $n$ is odd, then $\{p_1^{(2)}, \ldots, p_{n-1}^{(2)}, \Pf\}$ forms a homogeneous basis for $\VV$, so we know a priori that they will be eigenvectors. We wish to find the associated eigenweights $\epsilon_k(\Omega, \mu)$ and $\epsilon_{\Pf}(\Omega, \mu)$ such that 
\[
\ol {\nabla}^\eta_\mu (p_k^{(2)}) \equiv \epsilon_k(\Omega, \mu) p_k^{(2)} \pmod{I^2} \quad \text{ and } \quad \ol {\nabla}^\eta_\mu(\Pf) = \epsilon_{\Pf}(\Omega, \mu) \Pf \pmod{I^2}.
\]
When $n=2m$, this is still the case except in grading degree $2n$, where $\VV_{2n}$ is spanned by $p_m^{(2)}$ and $\Pf$. On $\VV_{2n}$, the action of $\ol {\nabla}^\eta_\mu$ has the form 
\begin{equation}
\begin{pmatrix} A_m & B_m \\ C_m & D_m \end{pmatrix}
\end{equation}
whose entries we will calculate explicitly. In particular, we will show that $B_m$ and $C_m$ can be non-zero, which implies that $\ol \nabla_{\omega_n^\vee}$ need not commute with $\ol \nabla_{\omega_1^\vee}$, addressing a question from \cite{FYZ4}.


\subsubsection{Integration}
We consider the integration map $\int_{G/P_\mu}$ applied to Schur polynomials in $X$.

\begin{lemma} \label{lem:integral}
Recall that $\delta_n = (n-1, n-2, \ldots, 0)$. Let $\lambda$ be a partition with at most $n$ parts.
\begin{enumerate}
    \item If $\lambda+\delta_n$ has entries of mixed parity, then 
    \[
    \int_{G/P_\mu} s_\lambda(X) = 0.
    \]
    \item \textbf{Case E (Even):} If $\lambda+\delta_n = 2\kappa$ for some partition $\kappa$, then
    \[ 
    \int_{G/P_\mu} s_\lambda(X) = (-1)^{D_\mu} 2^{n-1} s_\pi(X^{(2)}) 
    \]
    where $\pi := \kappa - \delta_n$. We write $\lambda_E(\pi) := 2\pi + \delta_n$.
    \item \textbf{Case O (Odd):} If $\lambda+\delta_n = 2\kappa + (1^n)$ for some partition $\kappa$, then
    \[ 
    \int_{G/P_\mu} s_\lambda(X) = (-1)^{D_\mu} 2^{n-1} \Pf s_\pi(X^{(2)})
    \]
    where $\pi := \kappa - \delta_n$.
We write $\lambda_O(\pi) := 2\pi + \delta_n + (1^n)$. 
\end{enumerate}
\end{lemma}

\begin{proof} We prepare to apply \cite[Lemma 5.4.4]{FYZ4}. We have $W/W_\mu \cong K$, the group of even sign changes of $X$. In particular, $|K| = 2^{n-1}$. 
We have
\[ 
\mathfrak{R}_\mu = (-1)^{D_\mu} \prod_{1 \leq i < j \leq n} (x_i + x_j).
\]
Let $\Delta(X) = \prod_{1 \leq i < j \leq n} (x_i - x_j)$ be the Vandermonde determinant. Using the bialternant formula \eqref{eq:bialternant}, \cite[Lemma 5.4.4]{FYZ4} says that
\begin{equation}\label{eq:D-proof-1}
\int_{G/P_\mu} s_\lambda(X) = \sum_{\sigma \in K} \sigma\left( \frac{a_{\lambda+\delta_n}(X)}{\Delta(X)\frR_\mu} \right).
\end{equation}

Observing that $\Delta(X^{(2)}) = (-1)^{D_\mu} \Delta(X)\frR_\mu$, we can rewrite \eqref{eq:D-proof-1} as 
\begin{equation}\label{eq:D-integration}
\int_{G/P_\mu} s_\lambda(X)  = \frac{(-1)^{D_\mu}}{\Delta(X^{(2)})} \sum_{\sigma \in K} \sigma(a_{\lambda+\delta_n}(X)). 
\end{equation}
Set $\gamma := \lambda+\delta_n$. The action of $\sigma = (\sigma_j)_{j=1}^n \in K$ on $a_\gamma(X)$ is 
\[
\sigma(a_\gamma(X)) = \sum_{w \in S_n} \mathrm{sgn}(w) (\prod_{j=1}^n \sigma_{w(j)}^{\gamma_j}) X^{w(\gamma)}.
\]
Therefore, the sum over $K$ vanishes unless the character $\sigma \mapsto \prod_j \sigma_j^{\gamma_j}$ is trivial on $K$, which occurs if and only if all $\gamma_j$ have the same parity. This proves assertion (1). 

If the $\gamma_j$ do all have the same parity, then we have $\sigma(a_\gamma(X)) = a_\gamma(X)$ for all $\sigma \in K$, so that 
\[
\sum_{\sigma \in K} \sigma(a_{\lambda+\delta_n}(X)) = 2^{n-1} a_{\lambda+\delta_n}(X).
\]

Suppose we are in Case $E$, with $\lambda + \delta_n = 2\kappa$. Then we may write 
\[
a_{\lambda+\delta_n}(X) = a_{2\kappa}(X) = a_{\kappa}(X^{(2)}) = a_{\pi + \delta_n}(X^{(2)}) \quad \text{for $\pi = \kappa-\delta_n$.}
\]
Putting this into \eqref{eq:D-integration} yields 
\[
\int_{G/P_\mu} s_{\lambda}(X) = (-1)^{D_\mu} 2^{n-1} \frac{a_{\pi + \delta_n}(X^{(2)})}{\Delta(X^{(2)})} = (-1)^{D_\mu} 2^{n-1} s_{\pi}(X^{(2)}).
\]  
For Case O, the result follows from a similar analysis, using that $\Pf = x_1 x_2 \cdots x_n  = e_n(X) =  X^{(1^n)}$. 
\end{proof}

\subsubsection{Action on $p_k^{(2)}$}

We analyze the action of $\ol {\nabla}^\eta_\mu$ on the generators $\{p_k^{(2)}\}_{1 \leq k \leq n-1}$ and $\Pf$ modulo $I^2$.
For a partition $\pi$ with $|\pi| = k > 0$, we have from Lemma \ref{lem:schur-to-power} that \begin{equation}\label{eq:schur-to-power-2}
s_\pi(X^{(2)}) \equiv \frac{\chi^\pi((k))}{k} p_k^{(2)} \pmod{I^2}.
\end{equation}
By Lemma \ref{lem:char-val-MN}, $\chi^\pi((k))$ vanishes unless $\pi = \pi_j(k) = (k-j, 1^j)$ for $0 \leq j < k$, in which case we have $\chi^{\pi_j(k)}((k)) = (-1)^j$.

For $1 \leq k \leq n-1$, we have $\partial_\mu (p_k^{(2)}) = k p_{2k-1}(X)$. Then
\begin{equation}\label{eq:D-integrand}
\eta \partial_\mu p_k^{(2)} = 2^{-N}  k p_1(X)^{N} p_{2k-1}(X)  = 2^{-N}  k p_{\nu_k}(X),
\end{equation}
where $\nu_k = (2k-1, 1^{N})$ is a partition of $N+2k-1$. By the Frobenius character formula \eqref{eq:Frob-char-formula}, we can rewrite  
\[
p_{\nu_k}(X)  =  \sum_{|\lambda|=N+2k-1} \chi^\lambda(\nu_k) s_\lambda(X).
\]
Substituting this into \eqref{eq:D-integrand}, we have
\begin{equation}\label{eq:D-proof-2}
\ol {\nabla}^\eta_\mu(p_k^{(2)})  =  2^{-N}  k  \sum_{|\lambda|=N+2k-1} \chi^\lambda (\nu_k) \int_{G/P_\mu} s_\lambda(X).
\end{equation}
We use Lemma \ref{lem:integral} to evaluate $\int s_\lambda(X)$. By Lemma \ref{lem:integral}(1), the only contribution is from $\lambda$ such that $\lambda+\delta_n$ is in either Case E or Case O.

\begin{enumerate}
\item[Case E] If $\lambda = 2\pi + \delta_n$, then from $|\lambda| = \binom{n}{2}+2k$ we must have $|\pi| =  k$. In this case, applying Lemma \ref{lem:integral} implies that the contribution of $\lambda$ to \eqref{eq:D-proof-2} is
\begin{equation}\label{eq:D-even-contribution}
(-1)^{D_\mu} 2^{n-1-N} k \chi^\lambda (\nu_k) s_\pi(X^{(2)}).
\end{equation}

Using the Frobenius character formula \eqref{eq:schur-to-power-2} and Lemma \ref{lem:char-val-MN}, we find that the contribution from Case E to \eqref{eq:D-proof-2} is 
\begin{equation}\label{eq:D-eps-k-reappear}
\Big((-1)^{D_\mu} k 2^{n-1-N} \sum_{|\pi| = k} \chi^{\lambda_E(\pi)}(\nu_k) \frac{\chi^\pi((k))}{k}\Big) p_k^{(2)} = \Big((-1)^{D_\mu} 2^{n-1-N}\sum_{j=0}^{k-1} (-1)^j \chi^{2 \pi_j(k)+ \delta_n}(\nu_k)\Big) p_k^{(2)}.
\end{equation}

\item[Case O] If $\lambda = 2\pi + \delta_n + (1^n)$, then from $|\lambda| = \binom{n}{2}+2k$ we must have $2|\pi| = 2k - n$. This is only possible if $n$ is even, say $n=2m$, and $k \geq m$. Lemma \ref{lem:integral} asserts that the resulting term is a scalar multiple of $\Pf s_\pi(X^{(2)})$. Modulo $I^2$, this vanishes unless $|\pi|=0$ (as $\Pf \in I$). In that non-vanishing case, we must have $k=m$, $\pi = \emptyset$, and $\lambda = \lambda_O(\emptyset) = \delta_{n} + (1^n)$. We then have
\[ 
\ol {\nabla}^\eta_\mu(p_m^{(2)}) \equiv A_m p_m^{(2)} + C_m \Pf \pmod{I^2}, 
\]
where as in \eqref{eq:D-eps-k-reappear} we have
\[
 A_m = (-1)^{D_\mu} 2^{n-1-N}\sum_{j=0}^{m-1} (-1)^j \chi^{2 \pi_j(m)+ \delta_n}(\nu_m),
\]
and
\begin{equation}\label{eq:c_m}
C_m = (-1)^{D_\mu} m 2^{n-1-N} \chi^{\delta_{n} + (1^n)}(\nu_m).
\end{equation}
\end{enumerate}

\subsubsection{Action on $\Pf$} Next we consider $\ol {\nabla}^\eta_\mu(\Pf)$. We have 
\[
\partial_\mu (\Pf) = \frac{1}{2} e_{n-1}(X) = \frac{1}{2}s_{(1^{n-1})}(X).
\]
Thus we are contemplating 
\begin{equation}\label{eq:f-nabla}
\ol {\nabla}^\eta_\mu(\Pf) = 2^{-N-1} \int_{G/P_\mu} p_1(X)^{N} s_{(1^{n-1})}(X).
\end{equation}

We may view $p_1(X) = s_{(1)}(X)$. The Littlewood--Richardson rule gives a formula for the product of two Schur polynomials; we need it only for the case where one partition has a single column, which is given by \emph{Pieri's rule} \cite[Theorem 7.15.7]{Stan2}. Iterating it yields
\[
p_1(X)^N s_\nu(X) = \sum_{\substack{\lambda \supseteq \nu \\ |\lambda|=|\nu|+N}} f^{\lambda/\nu} s_\lambda(X),
\]
where $f^{\lambda/\nu}$ is the number of Standard Young Tableaux (SYT) of skew shape $\lambda/\nu$ with size $N$ (for the definitions, see \cite[\S 7.10]{Stan2}). In these terms, the integrand of \eqref{eq:f-nabla} may be written as 
\[
p_1(X)^{N} s_{(1^{n-1})}(X) = \sum_{\substack{\lambda \supseteq (1^{n-1}) \\ |\lambda|=n-1+N}} f^{\lambda/(1^{n-1})} s_\lambda(X).
\]
Hence we have 
\[
\ol {\nabla}^\eta_\mu(\Pf)  =  2^{-N-1} \sum_{\substack{\lambda \supseteq (1^{n-1}) \\ |\lambda|= \binom{n}{2}+n }} f^{\lambda/(1^{n-1})}  \int_{G/P_\mu} s_\lambda (X).
\]
We analyze the inner integral using Lemma \ref{lem:integral}. The only contributions come from Case E and Case O, and we consider each in turn. 

\begin{enumerate}
\item[Case O]
Since $|\lambda| = \binom{n+1}{2}$, we must have $|\pi| = 0$. Thus the only possibility is $\lambda = \delta_n + (1^n)$. Let us call $\rho_n := \delta_n + (1^n)$. In this case, 
the skew shape $\lambda / (1^{n-1})$ is a disjoint union of the single cell $(n,1)$ and the shifted staircase $\rho_{n-1}$. Therefore, we have 
\[
f^{\rho_n / (1^{n-1})} = \Big(\binom{n}{2}+1\Big) f^{\rho_{n-1}}.
\]

The dimension of the irreducible character $\chi^\lambda$ of $S_N$ is given by the \emph{Hook Length Formula} \cite[Theorem 10.3.2 and Theorem 2.5.2]{Sag01}
\begin{equation}\label{eq:hook-length-formula}
\chi^\lambda(\Id) = \dim \chi^\lambda = \frac{N!}{\displaystyle\prod_{(i,j) \in \text{Young diagram}(\lambda)} h_{\lambda}(i,j)}
\end{equation}
For the partition $\rho_{n-1}$, the Hook Length Formula gives 

\begin{minipage}{0.48\textwidth}
  \centering
  \[
  f^{\rho_{n-1}} = \frac{ \binom{n}{2}!}{\prod_{i=1}^{n-1}(2i-1)^{n-i}}
  \]
\end{minipage}
\hfill
\begin{minipage}{0.48\textwidth}
    \centering
    \scalebox{0.7}{
        \ytableausetup{boxsize=1.8em}
        \begin{ytableau}
            {\scriptstyle 2n-3} & \dots & 5 & 3 & 1 \\
            \vdots & 5 & 3 & 1 \\
            5 & 3 & 1 \\
            3 & 1 \\
            1
        \end{ytableau}
    }
    \captionof{figure}{Hook lengths for the partition $\rho_{n-1}$}
    \label{fig:hook_lengths_rho}
\end{minipage}

Hence the contribution to \eqref{eq:f-nabla} from Case O is 
\begin{equation}\label{eq:D-eps-pfaff}
(-1)^{D_\mu}  2^{n-2-N} \Big(\binom{n}{2}+1\Big) \frac{ \binom{n}{2}!}{\prod_{i=1}^{n-1}(2i-1)^{n-i}}  (\Pf). 
\end{equation}

\item[Case E]
This occurs only if $n=2m$ is even. Since $|\lambda| = \binom{n}{2} + n = 2|\pi| + \binom{n}{2}$, we see that $|\pi| = m$. By \eqref{eq:schur-to-power-2} and the paragraph following it, $s_\pi(X^{(2)}) = 0 \pmod{I^2}$ unless $\pi = \pi_j(m) = (m-j,1^j)$ for $0 \leq j < m$, so the contribution of Case E to $\ol {\nabla}^\eta_\mu(\Pf)$ is
\begin{equation}
\Big((-1)^{D_\mu} \frac{2^{n-2-N}}{m} \sum_{j=0}^{m-1} (-1)^j f^{(2\pi_j(m)+\delta_n) /(1^{n-1})} \Big) p_m^{(2)}. 
\end{equation}
\end{enumerate}

\subsubsection{Summary}

If $n$ is odd, then we have 
\[
\epsilon_k(\Omega, \mu) = (-1)^{D_\mu} 2^{n-1-N} \sum_{j=0}^{k-1} (-1)^j \chi^{2\pi_j(k) + \delta_n}(\nu_k)
\]
and 
\[
\epsilon_{\Pf}(\Omega, \mu) = (-1)^{D_\mu} 2^{n-2-N} \Big(\binom{n}{2}+1\Big) \frac{ \binom{n}{2}!}{\prod_{i=1}^{n-1}(2i-1)^{n-i}}.
\]
If $n=2m$ is even, the same formulas hold for $\epsilon_k$ when $k \neq m$, while the action of $\ol {\nabla}^\eta_\mu$ on the ordered basis $(p_m^{(2)}, \Pf)$ for $\VV_{2n}$ is given by the matrix 
\begin{equation}\label{eq:D-even-matrix}
(-1)^{D_\mu} 2^{n-1-N} \begin{pmatrix}
 \sum_{j=0}^{m-1} (-1)^j \chi^{2\pi_j(m) + \delta_n}(\nu_m) & \frac{1}{n} \sum_{j=0}^{m-1} (-1)^j f^{(2\pi_j(m)+\delta_n) /(1^{n-1})} \\
m  \chi^{\delta_{n} + (1^n)}(\nu_m) & \frac{1}{2} \Big(\binom{n}{2}+1\Big) \frac{ \binom{n}{2}!}{\prod_{i=1}^{n-1}(2i-1)^{n-i}}
\end{pmatrix}.
\end{equation}
We have completed the proof of Theorem \ref{thm:D-main}. \qed 

\subsection{Examples}\label{ssec:D-n=4}
For $G = \PSO_8$, we will calculate the off-diagonal entries of the matrix \eqref{eq:D-even-matrix}, and in particular show that they are both non-vanishing. Here $m=2$ and $ \nu_m = (3, 1^7)$. Note that $(-1)^{N-1} = (-1)^{6}= 1$ in this case. 

The lower left entry is 
\[
C_2 = 2^{-4} \cdot 2 \cdot \chi^{(4,3,2,1)}((3, 1^7)).
\]
In Example \ref{ex:MN-rule-2}, we computed $\chi^{(4,3,2,1)}((3, 1^7)) = -48$, so we find that 
\[
C_2 = 2^{-3} \cdot (-48) = \boxed{-6.}
\]

Next we consider the upper right entry,  
\[ 
B_2 = 2^{-4} \cdot \frac{1}{4} \cdot \left( (-1)^0 f^{(2\pi_0(2)+\delta_4)/(1^3)} + (-1)^1 f^{(2\pi_1(2)+\delta_4)/(1^3)} \right).
\]
\begin{itemize}
\item We have $2\pi_0(2) + \delta_4 = 2(2) + (3,2,1,0)  = (7,2,1)$. The skew shape $(7,2,1)/(1,1,1)$ is the partition $(6,1)$. By the Hook Length formula (see figure below), we have 
\[
f^{(7,2,1)/(1,1,1)} = f^{(6,1)} = 6.
\]
\[
\ytableausetup{boxsize=1.5em} 
\begin{ytableau}
  7  & 5 & 4 & 3 & 2 & 1 \\
  1
\end{ytableau} \hspace{1cm}
\ytableausetup{boxsize=1.5em} 
\begin{ytableau}
  5 & 4 & 3 & 1 \\
  3 & 2 & 1
\end{ytableau}
\]
\item We have $2\pi_1(2)+\delta_4=2(1,1) + (3,2,1,0) = (5,4,1)$. The skew shape $(5,4,1)/(1,1,1)$ is the partition $(4,3)$. By the Hook Length formula (see figure above), 
we have 
\[
f^{(5,4,1)/(1,1,1)} = f^{(4,3)} = 14.
\]
\end{itemize}
In conclusion, we find that 
\[
B_2 = 2^{-6} (f^{(7,2,1)/(1,1,1)} - f^{(5,4,1)/(1,1,1)} ) = 2^{-6} \cdot (6-14) = \boxed{-1/8.} 
\]
Continuing in this way, one finds that in this case \eqref{eq:D-even-matrix} is 
\[
\ol \nabla^\eta_\mu= \begin{pmatrix}	
5/2 & - 1/ 8 \\ 
-6 & 7/2
\end{pmatrix}
\]
The eigenweights are $4$ and $2$. As a sanity check, this matches the results \eqref{eq:D_n-standard-eigenweights}, which must be the case by the triality automorphism of the $D_4$ root system. 

\appendix

\bibliographystyle{amsalpha}
\bibliography{Bibliography}

\end{document}